\documentclass[runninghead]{llncs}

\usepackage{graphicx}
\usepackage{pgf,tikz}
\usetikzlibrary{arrows}
\usepackage{amsmath}
\usepackage{amsfonts}
\usepackage[hidelinks]{hyperref}

\let\oldendproof\endproof
\renewcommand\endproof{~\hfill$\qed$\oldendproof}

\newcommand{\Or}{O}        
\newcommand{\godd}{G^{odd}}      
\newcommand{\lw}{1.3pt}
\newcommand{\pw}{2.0pt}

\begin{document}
\title{Odd wheels are not odd-distance graphs}
\author{G\'abor Dam\'asdi\inst{1}\orcidID{0000-0002-6390-5419}}
\authorrunning{G. Dam\'asdi}
\institute{ MTA-ELTE Lend\"ulet Combinatorial Geometry Research Group \\ \email{damasdigabor@caesar.elte.hu} }
\maketitle              % typeset the header of the contribution
\begin{abstract}
 An odd wheel graph is a graph formed by connecting a new vertex to all vertices of an odd cycle. We answer a question of Rosenfeld and Le by showing that odd wheels cannot be drawn in the plane such that the lengths of the edges are odd integers.

\end{abstract}

\section{Introduction}

 A \emph{geometric graph} is a graph drawn in the plane so that the vertices are represented by distinct points and the edges are represented by possibly intersecting straight line segments connecting the corresponding points. A \emph{unit-distance graph} is a geometric graph where all edges are represented by segments of length 1. The study of unit-distance graphs started with the question of Edward Nelson, who raised the problem of determining the minimum number of colors that are needed to color the points of the plane so that no two points unit distance apart are assigned the same color. This number is known as the chromatic number of the plane. Until recently the best lower bound was 4, but Aubrey de Grey \cite{MR3820926} constructed a unit-distance graph which cannot be colored with $4$ colors. The best upper bound is 7. For more details on unit-distance graphs see for example \cite{MR2458293}. 
 
 Erd\H{o}s \cite{MR15796}   raised the problem to determine the maximal number of edges in a unit-distance graph with $n$ vertices and this question became known as the Erd\H{o}s Unit Distance Problem.
 
 Later Erd\H{o}s and Rosenfeld \cite{MR2519871} asked the same two questions for odd distances. Namely, let $\godd$ be the graph whose vertex set is the plane and two vertices are connected if their distance is an odd integer. They asked to determine the chromatic number of $\godd$, and to determine how many distances among $n$ points in the plane can be odd integers.

 Analogously we define \emph{odd-distance graphs} to be the geometric graphs having an embedding in the Euclidean plane in which  all edges are of odd integer length. In other words, the odd-distance graphs are the finite subgraphs of $\godd$. There are odd-distance graphs whose chromatic number is five \cite{MR2519871,MR3820926} but contrary to the unit distance case we do not have any upper bound. The chromatic number of $\godd$ might be infinite. In the case when we require the color classes to be measurable sets, it has been shown that the chromatic number is indeed infinite \cite{MR2438995,MR2505106}.    
 
  Four points in the plane with pairwise odd integer distances do not exist, hence $K_4$ is not an odd-distance graph. From Tur\'an's theorem we know that the complete tripartite graph $K_{n,n,n}$ has the maximal number of edges among $K_4$-free graphs.  Piepemeyer \cite{MR1397789} showed that $K_{n,n,n}$, and therefore any 3-colorable graph, is an odd-distance graph. This settles the second question of Erd\H{o}s and Rosenfeld. 
  
  Let $W_n$ be the wheel graph formed by connecting a new vertex to all vertices of a cycle on $n$ vertices.\footnote{There is some discrepancy in the literature, since some authors prefer to denote by $W_n$ the wheel graph on $n$ vertices.} The wheel graph $W_{2k}$ is 3-colorable, hence it is an odd-distance graph.

 Rosenfeld and Le \cite{MR3159068} showed that having $K_4$, which is also $W_3$, as a subgraph is not the only obstruction for being an odd-distance graph, since $W_5$ is also not an odd-distance graph. This led them to the following question: Is it true that $W_{2k+1}$ is  not a subgraph of $\godd$ for any $k$? We answer this for the affirmative.

\begin{theorem}\label{thm:mainwheel}
	The odd wheels are not odd-distance graphs. 
\end{theorem}

In Section \ref{section:observ} we consider drawings of wheel graphs in general, not assuming that the edge lengths are odd numbers. We develop a number of useful lemmas, that  might prove useful for related questions. For example Harborth's conjecture asks whether all planar graphs admit a planar drawing with integer edge lengths. Since the maximal planar graphs are the triangulations, they contain many wheels. Hence, understanding the possible drawings of wheels is vital for solving the conjecture. Then, in Section \ref{section:mainwheel}, we prove Theorem \ref{thm:mainwheel}.

\section{Wheels with integer edge lengths}\label{section:observ}

\subsection*{Embeddings of wheel graphs}

Every set of $n+1$ ordered points of the plane determines an \emph{embedding} of the wheel graph $W_n$. Throughout this paper we will always assume that the center of the wheel is embedded at the origin $\Or$ and the other points are $A_1,A_2,\dots, A_n$, following the order of the vertices in the defining cycle of the wheel (see Figure \ref{fig:notation}). In the following notations the index is understood cyclically, i.e. the index $n+1$ is equivalent to the index $1$.  For example every embedding determines $n$ triangles: $OA_iA_{i+1}$ for $i\in  \{1,\dots,n\}$. These will be referred to as the triangles of the embedding. We will use the following notations:

	\begin{equation*}\label{formula:aiaj}
	r_i=|\Or A_i|, \;\;\;\; r_{i,i+1}=|A_iA_{i+1}|
	\end{equation*}
	\begin{equation*}\label{formula:theta}
	\theta_{i,i+1}=\angle A_i\Or A_{i+1} 
	\end{equation*}

 That is, the $i$-th triangle has sides of length $r_i,r_{i+1}$ and $r_{i,i+1}$, and its inner angle is $\theta_{i,i+1}$. Note that the $\theta_{i,i+1}$-s are directed angles. We do not assume planarity or even general position of the points. For example crossings are allowed, and $\Or$ does not need to be in the interior of the cycle (see Figure \ref{fig:notation}). 

	\begin{figure}[ht!]
		\centering
			\begin{tikzpicture}[line cap=round,line join=round,>=triangle 45,x=0.5cm,y=0.5cm]
		\clip(-1.5057950876280606,-1.4925604465535915) rectangle (13.019376431358536,9.769100962850517);
%		\draw [line width=\lw] (0.,0.)-- (5.111111111111111,3.142696805273545);
		\draw [line width=\lw] (5.111111111111111,3.142696805273545)-- (9.,0.);
		\draw [line width=\lw] (5.111111111111111,3.142696805273545)-- (11.105549368960405,3.4009795434360437);
		\draw [line width=\lw] (11.105549368960405,3.4009795434360437)-- (9.,0.);
		\draw [line width=\lw] (5.111111111111111,3.142696805273545)-- (6.28086763452817,6.967833361780154);
		\draw [line width=\lw] (6.28086763452817,6.967833361780154)-- (11.105549368960405,3.4009795434360437);
		\draw [line width=\lw] (5.111111111111111,3.142696805273545)-- (2.6319521776883823,8.60655725213987);
		\draw [line width=\lw] (2.6319521776883823,8.60655725213987)-- (6.28086763452817,6.967833361780154);
		\draw [line width=\lw] (2.6319521776883823,8.60655725213987)-- (0.,0.);
		\draw [line width=\lw] (0.,0.)-- (9.,0.);
		\draw (6.458407149791733,7.98064853058783) node[anchor=north west] {$A_1$};
		\draw (1.7078303765939615,9.7) node[anchor=north west] {$A_2$};
		\draw (-0.8630699947836561,0.016446293168050904) node[anchor=north west] {$A_3$};
		\draw (11.3,4.3) node[anchor=north west] {$A_5$};
		\draw (9.3366977829645,0.2679474164549913) node[anchor=north west] {$A_4$};
		\draw (5.927460333963746,3.25) node[anchor=north west] {$\theta_{4,5}$};
		\draw (3.5,3.9286859887426795) node[anchor=north west] {$\theta_{2,3}$};
		\draw (4.25,5.4) node[anchor=north west] {$\theta_{1,2}$};
		\draw (4.6,2.9) node[anchor=north west] {$\theta_{3,4}$};
		\draw (5.6200700721685966,4.43168823531656) node[anchor=north west] {$\theta_{5,1}$};
		\draw [line width=\lw] (6.28086763452817,6.967833361780154)-- (5.111111111111111,3.142696805273545);
		\draw [line width=\lw] (5.111111111111111,3.142696805273545)-- (11.105549368960405,3.4009795434360437);
		\draw [line width=\lw] (11.105549368960405,3.4009795434360437)-- (6.28086763452817,6.967833361780154);
		\draw [line width=\lw] (5.111111111111111,3.142696805273545)-- (0.,0.);
		\draw [line width=\lw] (0.,0.)-- (9.,0.0001);
		\draw [line width=\lw] (9.,0.)-- (5.111111111111111,3.142696805273545);
		\draw [line width=\lw] (2.6319521776883823,8.60655725213987)-- (6.28086763452817,6.967833361780154);
		\draw [line width=\lw] (6.28086763452817,6.967833361780154)-- (5.111111111111111,3.142696805273545);
		\draw [line width=\lw] (5.111111111111111,3.142696805273545)-- (2.6319521776883823,8.60655725213987);
		\draw [line width=\lw] (2.6319521776883823,8.60655725213987)-- (0.,0.);
		\draw [line width=\lw] (0.,0.)-- (5.111111111111111,3.142696805273545);
		\draw [line width=\lw] (5.111111111111111,3.142696805273545)-- (2.6319521776883823,8.60655725213987);
		\draw [line width=\lw] (5.111111111111111,3.142696805273545)-- (11.105549368960405,3.4009795434360437);
		\draw [line width=\lw] (11.105549368960405,3.4009795434360437)-- (9.,0.);
		\draw [line width=\lw] (9.,0.)-- (5.111111111111111,3.142696805273545);
		\begin{scriptsize}
		\draw [fill=black] (0.,0.) circle (\pw);
		\draw [fill=black] (9.,0.) circle (\pw);
		\draw [fill=black] (5.111111111111111,3.142696805273545) circle (\pw);
		\draw [fill=black] (11.105549368960405,3.4009795434360437) circle (\pw);
		\draw [fill=black] (6.28086763452817,6.967833361780154) circle (\pw);
		\draw [fill=black] (2.6319521776883823,8.60655725213987) circle (\pw);
		\end{scriptsize}
		\end{tikzpicture}
        \begin{tikzpicture}[line cap=round,line join=round,>=triangle 45,x=0.4cm,y=0.4cm]
        \clip(-0.2251651062581305,-1.29014329240519865) rectangle (11.363939868228318,10.719416807231726);
        \draw (1.7011149489750954,9.00232945224248) node[anchor=north west] {$A_1$};
        \draw (10.153888482940524,6.948558302157305) node[anchor=north west] {$A_2$};
        \draw (8.199112194326565,1.4269440625840464) node[anchor=north west] {$A_3$};
        \draw (9.916199549315813,2.874684381496547) node[anchor=north west] {$A_5$};
        \draw (6.919713772962027,10.483738150664575) node[anchor=north west] {$A_4$};
        \draw [line width=\lw] (7.896096778740233,0.8882499904305499)-- (6.549361598356508,9.44001838586718);
        \draw [line width=\lw] (6.549361598356508,9.44001838586718)-- (9.613184133729483,2.3023219298334574);
        \draw [line width=\lw] (9.613184133729483,2.3023219298334574)-- (2.8121714727916696,7.722931030877936);
        \draw [line width=\lw] (2.8121714727916696,7.722931030877936)-- (10.,6.);
        \draw [line width=\lw] (7.896096778740233,0.8882499904305499)-- (10.,6.);
        \draw [line width=\lw] (0.,0.)-- (7.896096778740233,0.8882499904305499);
        \draw [line width=\lw] (0.,0.)-- (6.549361598356508,9.44001838586718);
        \draw [line width=\lw] (0.,0.)-- (9.613184133729483,2.3023219298334574);
        \draw [line width=\lw] (10.,6.)-- (0.,0.);
        \draw [line width=\lw] (0.,0.)-- (2.8121714727916696,7.722931030877936);
        \begin{scriptsize}
        \draw [fill=black] (0.,0.) circle (\pw);
        \draw [fill=black] (6.549361598356508,9.44001838586718) circle (\pw);
        \draw [fill=black] (9.613184133729483,2.3023219298334574) circle (\pw);
        \draw [fill=black] (2.8121714727916696,7.722931030877936) circle (\pw);
        \draw [fill=black] (10.,6.) circle (\pw);
        \draw [fill=black] (7.896096778740233,0.8882499904305499) circle (\pw);
        \end{scriptsize}
        \end{tikzpicture}

		\caption{Two embeddings of the wheel graph $W_5$.}
		\label{fig:notation}
	\end{figure}
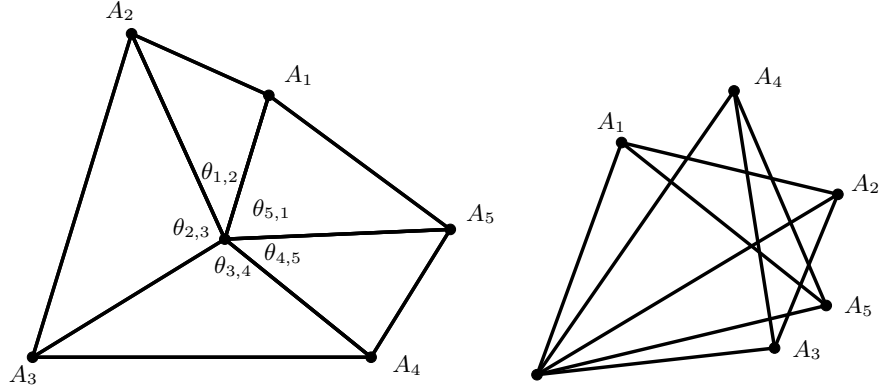

\subsection*{Geometry of a triangle}
Let us recall some classical results from elementary geometry. Let $T(a,b,c)$ denote a triangle with sides $a,b,c$ and angles $\alpha, \beta,\gamma$. By the law of cosines:

	\begin{equation}\label{formula:cos}
	\cos(\alpha)=\frac{b^2+c^2-a^2}{2bc}
	\end{equation}
	\begin{equation}\label{formula:sin}
	\sin(\alpha)=\sqrt{1-\cos(\alpha)^2}=\frac{\sqrt{4b^2 c^2 -(b^2 +c^2 -a^2)^2 }}{2bc}
	\end{equation}

Let $A$ denote the area of $T(a,b,c)$.

	\begin{equation}\label{formula:area}
	A=\frac{bc\sin(\alpha)}{2}=\frac{\sqrt{4b^2 c^2 -(b^2 +c^2 -a^2)^2 }}{4}
	\end{equation}

Using these formulas we will introduce two notions, the characteristic of a triangle and the residual of an angle. Strictly speaking we will only need residuals for the proof of Theorem \ref{thm:mainwheel}, but there is a strong connection to the characteristics of triangles so they are worth mentioning. 

\subsection*{Characteristic of a triangle}

From (\ref{formula:area}) we can see that if $a,b$ and $c$ are integers, then we can write the area of  $T(a,b,c)$ as $r \sqrt{D}$ for some rational number $r$ and a square-free integer $D$. If the area is $0$, then $r=0$ and $D$ can be any square-free integer. If the area is non-zero, then $D$ must be the square-free part of $4b^2 c^2 -(b^2 +c^2 -a^2)^2$. In this case the number $D$ is called the \emph{characteristic} of the triangle.  

We say that a point set in the plane is \emph{integral} if the pairwise distances are integers. The characteristic of triangles is a useful tool in the study and algorithmic generation of integral point sets (see for example \cite{MR2413160}). The following statement is folklore, for a proof see \cite{MR2262623}.    

\begin{theorem}
	The triangles spanned by each three non collinear points in a plane integral point set have the same characteristic. 
\end{theorem}

    Consider an embedding of a wheel graph where the edges have integer lengths. The rest of the distances might be non-integer, so the $n$ triangles of the embedding can have different characteristics. (See for example Figure \ref{fig:chars}). When we started the study of embeddings of wheel graphs, we hoped to show that there cannot be too many characteristics appearing in a embedding. It turned out that there can be arbitrarily many, but we can still gain some information by considering them.  Later in this section we will show the following statement.

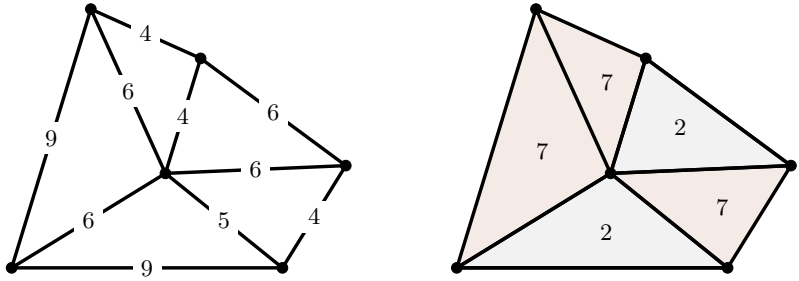
\begin{figure}[ht!]
	\centering
\definecolor{zzttqq}{rgb}{0.6,0.2,0.}
\definecolor{yqyqyq}{rgb}{0.5019607843137255,0.5019607843137255,0.5019607843137255}
\begin{tikzpicture}[line cap=round,line join=round,>=triangle 45,x=0.40cm,y=0.40cm]
\clip(-1.5057950876280606,-1.4925604465535915) rectangle (13.019376431358536,9.769100962850517);
\draw [line width=\lw] (0.,0.)-- (5.111111111111111,3.142696805273545) node [midway,fill=white] {6};
\draw [line width=\lw] (5.111111111111111,3.142696805273545)-- (9.,0.) node [midway,fill=white] {5};
\draw [line width=\lw] (5.111111111111111,3.142696805273545)-- (11.105549368960405,3.4009795434360437) node [midway,fill=white] {6};
\draw [line width=\lw] (11.105549368960405,3.4009795434360437)-- (9.,0.) node [midway,fill=white] {4};
\draw [line width=\lw] (5.111111111111111,3.142696805273545)-- (6.28086763452817,6.967833361780154) node [midway,fill=white] {4};
\draw [line width=\lw] (6.28086763452817,6.967833361780154)-- (11.105549368960405,3.4009795434360437) node [midway,fill=white] {6};
\draw [line width=\lw] (5.111111111111111,3.142696805273545)-- (2.6319521776883823,8.60655725213987) node [midway,fill=white] {6};
\draw [line width=\lw] (2.6319521776883823,8.60655725213987)-- (6.28086763452817,6.967833361780154) node [midway,fill=white] {4};
\draw [line width=\lw] (2.6319521776883823,8.60655725213987)-- (0.,0.) node [midway,fill=white] {9};
\draw [line width=\lw] (0.,0.)-- (9.,0.0001) node [midway,fill=white] {9};
\begin{scriptsize}
\draw [fill=black] (0.,0.) circle (\pw);
\draw [fill=black] (9.,0.) circle (\pw);
\draw [fill=black] (5.111111111111111,3.142696805273545) circle (\pw);
\draw [fill=black] (11.105549368960405,3.4009795434360437) circle (\pw);
\draw [fill=black] (6.28086763452817,6.967833361780154) circle (\pw);
\draw [fill=black] (2.6319521776883823,8.60655725213987) circle (\pw);
\end{scriptsize}
\end{tikzpicture}
\begin{tikzpicture}[line cap=round,line join=round,>=triangle 45,x=0.40cm,y=0.40cm]
\clip(-1.5057950876280606,-1.4925604465535915) rectangle (13.019376431358536,9.769100962850517);
\fill[line width=\lw,color=yqyqyq,fill=yqyqyq,fill opacity=0.10000000149011612] (6.28086763452817,6.967833361780154) -- (5.111111111111111,3.142696805273545) -- (11.105549368960405,3.4009795434360437) -- cycle;
\fill[line width=\lw,color=yqyqyq,fill=yqyqyq,fill opacity=0.10000000149011612] (5.111111111111111,3.142696805273545) -- (0.,0.) -- (9.,0.) -- cycle;
\fill[line width=\lw,color=zzttqq,fill=zzttqq,fill opacity=0.10000000149011612] (2.6319521776883823,8.60655725213987) -- (6.28086763452817,6.967833361780154) -- (5.111111111111111,3.142696805273545) -- cycle;
\fill[line width=\lw,color=zzttqq,fill=zzttqq,fill opacity=0.10000000149011612] (2.6319521776883823,8.60655725213987) -- (0.,0.) -- (5.111111111111111,3.142696805273545) -- cycle;
\fill[line width=\lw,color=zzttqq,fill=zzttqq,fill opacity=0.10000000149011612] (5.111111111111111,3.142696805273545) -- (11.105549368960405,3.4009795434360437) -- (9.,0.) -- cycle;
\draw [line width=\lw] (0.,0.)-- (5.111111111111111,3.142696805273545);
\draw [line width=\lw] (5.111111111111111,3.142696805273545)-- (9.,0.);
\draw [line width=\lw] (5.111111111111111,3.142696805273545)-- (11.105549368960405,3.4009795434360437);
\draw [line width=\lw] (11.105549368960405,3.4009795434360437)-- (9.,0.);
\draw [line width=\lw] (5.111111111111111,3.142696805273545)-- (6.28086763452817,6.967833361780154);
\draw [line width=\lw] (6.28086763452817,6.967833361780154)-- (11.105549368960405,3.4009795434360437);
\draw [line width=\lw] (5.111111111111111,3.142696805273545)-- (2.6319521776883823,8.60655725213987);
\draw [line width=\lw] (2.6319521776883823,8.60655725213987)-- (6.28086763452817,6.967833361780154);
\draw [line width=\lw] (2.6319521776883823,8.60655725213987)-- (0.,0.);
\draw [line width=\lw] (0.,0.)-- (9.,0.);
\draw [line width=\lw] (6.28086763452817,6.967833361780154)-- (5.111111111111111,3.142696805273545);
\draw [line width=\lw] (5.111111111111111,3.142696805273545)-- (11.105549368960405,3.4009795434360437);
\draw [line width=\lw] (11.105549368960405,3.4009795434360437)-- (6.28086763452817,6.967833361780154);
\draw [line width=\lw] (5.111111111111111,3.142696805273545)-- (0.,0.);
\draw [line width=\lw] (0.,0.)-- (9.,0.0001);
\draw [line width=\lw] (9.,0.)-- (5.111111111111111,3.142696805273545);
\draw (6.93346482711151,5.214136174431486) node[anchor=north west] {$2$};
\draw (4.47434273275031,1.7210650176684248) node[anchor=north west] {$2$};

\draw (8.330693289816736,2.5594020952915595) node[anchor=north west] {$7$};
\draw (2.3505554694383655,4.43168823531656) node[anchor=north west] {$7$};
\draw (4.502287302004415,6.7) node[anchor=north west] {$7$};
\begin{scriptsize}
\draw [fill=black] (0.,0.) circle (\pw);
\draw [fill=black] (9.,0.) circle (\pw);
\draw [fill=black] (5.111111111111111,3.142696805273545) circle (\pw);
\draw [fill=black] (11.105549368960405,3.4009795434360437) circle (\pw);
\draw [fill=black] (6.28086763452817,6.967833361780154) circle (\pw);
\draw [fill=black] (2.6319521776883823,8.60655725213987) circle (\pw);
\end{scriptsize}
\end{tikzpicture}

\caption{Characteristics of the triangles in an embedding.}
\label{fig:chars}
\end{figure}

\begin{lemma}\label{thm:multofpi}
	Consider an embedding of a wheel graph with integer edge lengths. Then the angles among the $\theta_{i,i+1}$-s corresponding to the triangles of a given characteristic add up to an integer multiple of $\pi$.
\end{lemma}

\subsection*{Residual of an angle}

Considering (\ref{formula:cos}) and (\ref{formula:area}) we can see that the characteristic of a triangle transfers to the angles in the following sense. If a triangle that have integer sides have characteristic $D$, then the sine of its angles have the form $ q\sqrt{D}$ for some rational number $q$, and the cosines of the angles are rational.  Hence we will say that an \emph{angle $\theta$ has residual $D$} if $D$ is square-free, $\sin(\theta)=q\sqrt{D}$ for some rational number $q$, and furthermore $\cos(\theta)$ is rational. 

Most angles in general do not have any residual. Integer multiples of $\pi$ have residual $D$ for any square-free integer $D$, but other angles have at most one residual. If the residual is unique, it will be called \emph{the residual of the angle}. For example the residual of $\frac{\pi}{2}$ is 1, the residual of $\frac{\pi}{3}$ is 3, but $\frac{\pi}{6}$ does not have any residual.  Just as the characteristic of triangles, the residual of the angles is a useful tool, in \cite{MR3229123} it was used to find trisectible angles in triangles that have integer sides.

The trigonometric addition formulas $\sin(\theta+\phi)=\sin(\theta)\cos(\phi)+\sin(\phi)\cos(\theta)$ and $\cos(\theta+\phi)=\cos(\theta)\cos(\phi)-\sin(\theta)\sin(\phi)$ immediately imply that the set of angles that have residual $D$ are closed under addition. Also, for any $\phi$ the following angles have the same set of residuals: $\phi, -\phi, \pi+\phi, \pi-\phi$.

\subsection*{Angles whose squared trigonometric functions are rational}

Conway, Radin and Sadun \cite{MR1706614} studied angles whose squared trigonometric functions are rational. They said that $\theta$ is a \emph{pure geodetic angle} if the square of its sine is rational and they showed the following theorem.

\begin{theorem}[The Splitting Theorem \cite{MR1706614}]\label{thm:splitting}
	If the value of a rational linear combination of pure geodetic angles is a rational multiple of $\pi$, then so is the value of its
	restriction to those angles whose tangents are rational multiples of any given square root.
\end{theorem}
 
 Clearly, angles that have residual $D$ are pure geodetic angles and have tangents that are rational multiples of $\sqrt{D}$. Therefore, Theorem \ref{thm:splitting} applies to them, but we can even strengthen it in some sense. Note that in the next theorem we consider simple sums instead of rational linear combinations.

\begin{theorem}[The Splitting Theorem for angles that have residual]\label{thm:splitwheel}
	Let us consider some angles that have a residual. If the value of the sum of these angles is a rational multiple of $\pi$, then so is the value of its restriction to those angles that have a given residual. Furthermore, these restricted sums must add up to integer multiples of $\frac{\pi}{3}$ or $\frac{\pi}{2}$.
\end{theorem}

\begin{proof}
The first part is clear from Theorem \ref{thm:splitting}. For the second part we recall Niven's Theorem.

\begin{theorem}[Niven's theorem \cite{MR0080123}]\label{thm:niven}
	Consider the angles in range $0\le \theta \le \frac{\pi}{2}$. The only values of $\theta$ such that both $\frac{\theta}{\pi}$ and $\cos(\theta)$ are rational are $0,\frac{\pi}{3}$  and $\frac{\pi}{2}$.
\end{theorem}

Since angles corresponding to a given residual are closed under addition, the sum restricted to residual $D$ gives us an angle that has residual $D$. But angles that have residual $D$ have rational cosine so we can apply Theorem \ref{thm:niven} to the restricted sums.  
\end{proof}

Now we are ready to prove Lemma \ref{thm:multofpi}. Since triangles that have characteristic $D$ have angles that have residual $D$, it is enough to show the following residual version. 

\begin{lemma}\label{lemma:multofpiresi}
	Consider an embedding of a wheel graph with integer edge lengths. Then the angles among the $\theta_{i,i+1}$-s corresponding to a given residual add up to an integer multiple of $\pi$.
\end{lemma}

\begin{proof}
	We know that $\sum\limits^n_{i=1}\theta_{i,i+1}$ is an integer multiple of $2\pi$. Hence we can apply Theorem \ref{thm:splitwheel} for the angles $\theta_{i,i+1}$.  Suppose that the angles corresponding to a given residual $D$ add up to $\theta$. From Theorem \ref{thm:splitwheel} it is clear that $\theta$ is either an integer multiple of $\frac{\pi}{2}$ or an integer multiple of $\frac{\pi}{3}$. Let $\theta'=\theta \mod \pi$. Then  $\theta'=0,\frac{\pi}{3},\frac{2\pi}{3}$ or $\frac{\pi}{2}$.  Note that since $\theta$ is the sum of some angles that have residual $D$, it also has residual $D$. Hence $\theta'$ also has residual $D$. 
	
	Since $\sin(\frac{\pi}{3})=\sin(\frac{2\pi}{3})=\frac{\sqrt{3}}{2}$, we have $D=3$ for $\theta'=\frac{\pi}{3}$ and also for $\theta'=\frac{2\pi}{3}$. Similarly we have $D=1$ for $\theta'=\frac{\pi}{2}$.  	Therefore, if we group the terms of $\sum\limits^n_{i=1}\theta_{i,i+1}$ based on the residuals, every group will sum up to an integer multiple of $\pi$ except maybe the ones corresponding to $D=1$ and $D=3$. (Some $\theta_{i,i+1}$ might not have a unique residual but those are themselves integer multiples of $\pi$, we can just pick an arbitrary residual for them).  Since the whole sum should be an integer multiple of $\pi$, the exceptional cases together must add up to a integer multiple of $\pi$. This can only happen if both of them add up to an integer multiple of $\pi$, since $\frac{\pi}{3}+\frac{\pi}{2}$ and $\frac{2\pi}{3}+\frac{\pi}{2}$ are not integer multiples of $\pi$.  Hence, every sum corresponding to a given residual is a integer multiple of $\pi$.
\end{proof}

\section{Wheels with odd edge lengths}\label{section:mainwheel}

In the previous section we considered wheels with arbitrary integer edge lengths. Now we are ready to turn our attention to drawings of odd wheels where the edge lengths are odd numbers.

\begin{lemma}\label{lemma:3mod8}
	If $a,b$ and $c$ are odd numbers and the characteristic of the triangle $T(a,b,c)$ is $D$, then $D\equiv3 \mod 8$.
\end{lemma}
\begin{proof}
	From  (\ref{formula:area}) we know that the characteristic of the triangle is the square-free part of $4a^2 b^2 -(a^2 +b^2 -c^2)^2$. Since squares of odd numbers are congruent to 1 modulo 8, we have $4a^2 b^2 -(a^2 +b^2 -c^2)^2\equiv3\mod 8$. Since the square part of $4a^2 b^2 -(a^2 +b^2 -c^2)^2$ is the square of an odd number it is congruent to $1$ modulo $8$. Hence $D\equiv 3 \mod 8$.
\end{proof}

This means that if we have an embedding of a wheel graph where the edge lengths are odd integers, then each $\theta_{i,i+1}$ have a unique residual that is congruent to 3 modulo 8. The next idea is to classify the angles whose residual is congruent to 3 modulo 8.

\begin{lemma}\label{lemma:froms}
	Suppose $D\equiv 3 \mod 8$ and $\phi$ is an angle that has residual $D$. Then $\cos(\phi)$ can be written as $\frac{m}{2p}$, where $p\equiv 1 \mod 8$ and $m$ is an integer. Furthermore the remainder of $m$ modulo 8 is determined by the angle, and it is either $1,2,3,5,6$ or $7$. 
\end{lemma}

We will call this remainder the \emph{class} of $\phi$. 

\begin{proof}
	By the definition of having a residual $\cos(\phi)$ is rational. Since $D\equiv 3 \mod 8$ the value of $\cos(\phi)$ is non-zero. Hence, we can write $\cos(\phi)=\frac{a}{b}$ for some non-zero integers $a,b$ such that $gcd(a,b)=1$. There are two cases.
	
	First, suppose that $a$ and $b$ are odd. Odd numbers have an inverse in $\mathbb{Z}_8$. So, if $b$ is odd, there is an odd number $k$ such that $bk\equiv 1 \mod 8$. Hence we can write $\cos(\phi)=\frac{2ak}{2bk}$. Now $ak$ is odd, therefore $2ak$ is not divisible by 4. Hence $m=2ak,$ $p=bk$  works.

	Second, suppose that $a$ or $b$ is even. Since $gcd(a,b)=1$, one of them is even and the other one is odd. Consider that $\sin(\phi)=\pm\sqrt{1-\frac{a^2}{b^2}}=\pm\frac{\sqrt{b^2-a^2}}{b}$. The square-free part of $b^2-a^2$ is $D$, and $b^2-a^2$ is odd, so $b^2-a^2\equiv 3 \mod 8$. Since the only quadratic residuals modulo $8$ are $0$, $1$ and $4$, the only possibility is that $b^2\equiv 4 \mod 8$ and $a^2\equiv 1 \mod 8$. Since $b^2\equiv 4 \mod 8$,  $b'=\frac{b}{2}$ is odd and similarly to the previous case there is an odd $k$ such that $kb'\equiv 1\mod 8$. Since $a^2\equiv 1 \mod 8$, $a$ must be odd. So $m=ak$, $p=b'k$ works.

	It is also easy to see that an angle cannot fall into two classes, notice that $\frac{m_1}{2p_1}=\frac{m_2}{2p_2}$ implies $m_1p_2\equiv m_2p_1\mod 8$. 
\end{proof}

The aim of the next lemma is to answer the following question. Suppose we have two angles one of class $m_1$ and one of class $m_2$. If their sum have a class, what could that be?  

\begin{lemma}\label{lemma:classadition}
	If $\cos(\theta)=\frac{m_{1}}{2p_{1}}$, $\cos(\phi)=\frac{m_{2}}{2p_{2}}$ and $\cos(\theta+\phi)=\frac{m_{3}}{2p_{2}}$ for some integers $p_{1},p_{2},p_{3}$ that are congruent to 1 modulo 8 and integers $m_{1},m_{2},m_{3}$, then  
	
\begin{equation}\label{formula:cosadd}
m_{1}^2+m_{2}^2+m_{3}^2-m_{1}m_{2}m_{3}-4\equiv 0 \mod 8 
\end{equation} 

\end{lemma}

\begin{proof}

Using the cosine addition formula $\cos(\theta+\phi)=\cos(\theta)\cos(\phi)-\sin(\theta)\sin(\phi)$:
\begin{equation*}
	\frac{m_{3}}{2p_{3}}=\frac{m_{1}}{2p_{1}}\cdot\frac{m_{2}}{2p_{2}}-\left(\pm\frac{\sqrt{4p_{1}^2-m_{1}^2}}{2p_{1}}\right)\cdot\left(\pm\frac{\sqrt{4p_{2}^2-m_{2}^2}}{2p_{2}}\right)
\end{equation*}

\begin{equation*}
	(2m_{3}p_{1}p_{2}-m_{1}m_{2}p_{3})^2=p_{3}^2(4p_{1}^2-m_{1}^2)(4p_{2}^2-m_{2}^2)
\end{equation*}
\begin{equation*}
4m_{3}^2p_{1}^2p_{2}^2+m_{1}^2m_{2}^2p_{3}^2-4m_{1}m_{2}m_{3}p_{1}p_{2}p_{3}= 16p_{1}^2p_{2}^2p_{3}^2-4p_{1}^2m_{2}^2p_{3}^2-4m_{1}^2 p_{2}^2p_{3}^2+m_{1}^2m_{2}^2p_{3}^2
\end{equation*}
\begin{equation*}
p_{1}^2p_{2}^2m_{3}^2-p_{1}p_{2}p_{3}m_{1}m_{2}m_{3}-4p_{1}^2p_{2}^2p_{3}^2+ p_{1}^2m_{2}^2p_{3}^2+ m_{1}^2p_{2}^2p_{3}^2=0
\end{equation*}

Using that $p_1\equiv p_2 \equiv p_3\equiv 1 \mod 8$  we get (\ref{formula:cosadd}).

\end{proof}

Consider the solutions of (\ref{formula:cosadd}) in $\mathbb{Z}_8$. Clearly, every triple $(m_1,m_2,m_3)$ that is not a solution of this equation  encodes a forbidden change in the class when adding two angles. For example, since $(1,2,3)$ is not a solution, adding an angle of class 1 and an angle of class 2 cannot result in an angle of class 3.  The equation is symmetric in $m_{1},m_{2}$ and $m_{3}$. We will be later interested in solutions where one of the $m_i$-s is $1,3,5$ or $7$.  Checking every triple we find that these solutions are the following ones and the re-orderings of these: $(1, 1, 2)$,
	$(1, 1, 7)$,
	$(1, 2, 5)$,
	$(1, 3, 5)$,
	$(1, 3, 6)$,
	$(1, 6, 7)$,
	$(2, 3, 3)$,
	$(2, 3, 7)$,
	$(2, 5, 5)$,
	$(2, 7, 7)$,
	$(3, 3, 7)$,
	$(3, 5, 6)$,
	$(5, 5, 7)$,
	$(5, 6, 7)$,
	$(7, 7, 7)$.

\subsection*{Proof of main theorem}

The idea of the proof is simple, we want to show that $\sum\limits_{i=1}^n \theta_{i,i+1}$ is not a multiple of $2\pi$ using the fact that each $\theta_{i,i+1}$ is an angle of a triangle whose sides have odd length. This will not work this easily, for example both $\frac{\pi}{3}$ and $\frac{2\pi}{3}$ appears in triangles with odd sides and $\frac{\pi}{3}+\frac{\pi}{3}+\frac{\pi}{3}+\frac{\pi}{3}+\frac{2\pi}{3}=2\pi$. To reach a contradiction we will also use that the triangles in a wheel embedding share sides with their neighbours. 
  
\begin{proof}[of Theorem \ref{thm:mainwheel}]
    
Suppose there is a counterexample to Theorem \ref{thm:mainwheel}. From Lemma \ref{lemma:3mod8} we know that each $\theta_{i,i+1}$ has a unique residual. Let $\phi_1,\dots, \phi_{n}$ be a reordering of the angles $\theta_{1,2},\theta_{2,3},\dots, \theta_{n,1}$ in such a way that the angles of given residuals are consecutive. In general an arbitrary angle might not have any residual. The advantage of this ordering is that $\sum\limits_{i=1}^\ell \phi_i$ has a residual for each $\ell\in \{0,1,\dots,n\}$. To see this suppose that the residual of $\phi_\ell$ is $D$. From Lemma \ref{lemma:multofpiresi} we see that the $\phi_i$-s before $\phi_\ell$ whose residual is not $D$ sum up to an integer multiple of $\pi$. Thus, they do not affect the residual of $\sum\limits_{i=1}^\ell \phi_i $. Since angles that have residual $D$ are closed under addition, the rest sums up to an angle that has residual $D$. 

We also know that $D\equiv 3\mod 8$ from Lemma \ref{lemma:3mod8}. Hence by Lemma  \ref{lemma:froms} we know that  $\sum\limits_{i=1}^\ell \phi_i $ have a class for each $\ell\in \{0,1,\dots,n\}$. Consider how the class changes as $\ell$ goes from $0$ to $n$. 

In each step we increase the angle by some $\theta_{j,j+1}$. We have $\cos(\theta_{j,j+1})=\frac{r_j^2+r_{j+1}^2-r_{j,j+1}^2}{2r_jr_{j+1}}=\frac{(r_j^2+r_{j+1}^2-r_{j,j+1}^2)r_jr_{j+1}}{2r_j^2r_{j+1}^2}$. Since $r_j,r_{j+1}$ and $r_{j,j+1}$ are odd numbers, $(r_j^2+r_{j+1}^2-r_{j,j+1}^2) \equiv 1 \mod 8$ and $r_j^2r_{j+1}^2 \equiv 1 \mod 8$. Therefore the class of $\theta_{j,j+1}$ is the remainder of $r_jr_{j+1}$  by eight, which is either $1,3,5$ or $7$. We will use this fact in the following form. If $r_jr_{j+1}\equiv 1 \mod 4$, then the class of $\theta_{j,j+1}$ is either $1$ or $5$, and if $r_jr_{j+1}\equiv 3 \mod 4$, then the class of $\theta_{j,j+1}$ is either $3$ or $7$. Therefore, as we increase $\ell$ the angle $\sum\limits_{i=1}^\ell \phi_i$ changes either by an angle whose class is 1 or 5, or by an angle whose class is 3 or 7 depending on the remainder of $r_jr_{j+1}$ divided by 4.

Now we are ready to use Lemma \ref{lemma:classadition}. The solutions of (\ref{formula:cosadd}) have an underlying structure, that we can use. This is depicted in Figure \ref{fig:types}. We create a graph $G$ whose vertex set is $\{1,2,3,5,6,7\}$. For solutions of the form $(1,x,y)$ and $(5,x,y)$ we have connected $x$ and $y$ by a dashed edge. For solutions of the form $(3,x,y)$ and $(7,x,y)$ we have connected them by and solid edge, allowing loop edges. These two sets of edges are disjoint. Note that the dashed edges form a bipartite graph such that the solid edges connect vertices inside the two parts. This will allow us to use a parity argument, as any closed trail in this graph must use an even number of dashed edges.

	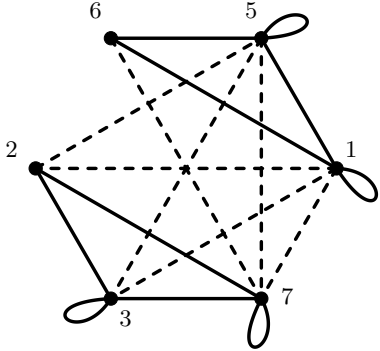
\begin{figure}[!ht]
		\centering
		\begin{tikzpicture}[line cap=round,line join=round,>=triangle 45,x=1.0cm,y=1.0cm]
		\clip(-1.6943074399048716,0.24406750691951654) rectangle (3.836255687938627,5.012309771195351);
		\draw (3.0,3.2) node[anchor=north west] {$1$};
		\draw (1.6688728405404991,5.05) node[anchor=north west] {$5$};
		\draw (-0.4,5.05) node[anchor=north west] {$6$};
		\draw (-1.5149378249477852,3.2) node[anchor=north west] {$2$};
		\draw (0.0,0.95) node[anchor=north west] {$3$};
		\draw (2.147191813759396,1.2455478570965728) node[anchor=north west] {$7$};
		\draw [line width=\lw] (-1.,2.732050807568879)-- (2.,1.);
		\draw [line width=\lw] (2.,1.)-- (0.,1.);
		\draw [line width=\lw] (0.,1.)-- (-1.,2.732050807568879);
		\draw [line width=\lw] (0.,4.464101615137755)-- (2.,4.4642);
		\draw [line width=\lw] (2.,4.464101615137755)-- (3.,2.7320508075688776);
		\draw [line width=\lw] (3.,2.7320508075688776)-- (0.,4.464101615137755);
		\draw [line width=\lw,dash pattern=on 3.0pt off 4.pt] (-1.,2.732050807568879)-- (2.,4.464101615137755);
		\draw [line width=\lw,dash pattern=on 3.0pt off 4.pt] (2.,4.464101615137755)-- (2.,1.);
		\draw [line width=\lw,dash pattern=on 3.0pt off 4.pt] (3.,2.7320508075688776)-- (-1.,2.732050807568879);
		\draw [line width=\lw,dash pattern=on 3.0pt off 4.pt] (2.,4.464101615137755)-- (0.,1.);
		\draw [line width=\lw,dash pattern=on 3.0pt off 4.pt] (3.,2.7320508075688776)-- (2.,1.);
		\draw [line width=\lw,dash pattern=on 3.0pt off 4.pt] (0.,4.464101615137755)-- (2.,1.);
		\draw [line width=\lw,dash pattern=on 3.0pt off 4.pt] (3.,2.7320508075688776)-- (0.,1.);
		\draw[line width=\lw,fill=black,fill opacity=0.10000000149011612] (-0.16966571276963688,-5.3014430888370185) -- (2.8198278698484707,-5.3014430888370185);
		\draw[line width=\lw] (0.,1.) -- (-0.009830031860551905,0.9845759104663253) -- (-0.01976372561904602,0.9696136308600719) -- (-0.029794927333832585,0.9551079490174387) -- (-0.03991748306326187,0.9410536527746246) -- (-0.05012523886568411,0.9274455299678286) -- (-0.060412040799449555,0.9142783684332494) -- (-0.07077173492290846,0.9015469560070861) -- (-0.08119816729441108,0.8892460805255373) -- (-0.09168518397230765,0.8773705298248022) -- (-0.10222663101494844,0.8659150917410796) -- (-0.1128163544806837,0.8548745541105682) -- (-0.12344820042786364,0.8442437047694669) -- (-0.13411601491483857,0.8340173315539746) -- (-0.14481364399995872,0.8241902223002905) -- (-0.15553493374157432,0.8147571648446131) -- (-0.16627373019803562,0.8057129470231414) -- (-0.17702387942769288,0.7970523566720745) -- (-0.1877792274888964,0.7887701816276108) -- (-0.19853362043999634,0.7808612097259496) -- (-0.209280904339343,0.7733202288032897) -- (-0.22001492524528668,0.7661420266958296) -- (-0.2307295292161775,0.7593213912397689) -- (-0.24141856231036585,0.752853110271306) -- (-0.25207587058620184,0.7467319716266396) -- (-0.26269530010203584,0.7409527631419691) -- (-0.2732706969162181,0.735510272653493) -- (-0.2837959070870988,0.7303992879974106) -- (-0.2942647766730282,0.7256145970099201) -- (-0.30467115173235654,0.7211509875272211) -- (-0.31500887832343416,0.7170032473855119) -- (-0.3252718025046113,0.7131661644209919) -- (-0.33545377033423807,0.7096345264698596) -- (-0.34554862787066476,0.706403121368314) -- (-0.3555502211722418,0.703466736952554) -- (-0.3654523962973192,0.7008201610587783) -- (-0.3752489993042474,0.6984581815231862) -- (-0.3849338762513765,0.6963755861819763) -- (-0.3945008731970569,0.6945671628713473) -- (-0.40394383619963875,0.6930276994274985) -- (-0.4132566113174724,0.6917519836866286) -- (-0.4224330446089079,0.6907348034849363) -- (-0.4314669821322957,0.6899709466586208) -- (-0.4403522699459859,0.6894552010438807) -- (-0.4490827541083289,0.689182354476915) -- (-0.4576522806776747,0.6891471947939227) -- (-0.4660546957123739,0.6893445098311024) -- (-0.4742838452707766,0.6897690874246535) -- (-0.482333575411233,0.6904157154107742) -- (-0.4901977321920934,0.6912791816256638) -- (-0.49787016167170794,0.6923542739055212) -- (-0.5053447099084271,0.6936357800865451) -- (-0.5126152229606008,0.6951184880049344) -- (-0.5196755468865797,0.6967971854968881) -- (-0.5265195277447136,0.6986666603986049) -- (-0.5331410115933531,0.700721700546284) -- (-0.5395338444908483,0.7029570937761241) -- (-0.5456918724955495,0.705367627924324) -- (-0.551608941665807,0.7079480908270827) -- (-0.5572788980599708,0.7106932703205989) -- (-0.5626955877363915,0.7135979542410716) -- (-0.5678528567534191,0.7166569304246998) -- (-0.572744551169404,0.7198649867076824) -- (-0.5773645170426963,0.7232169109262179) -- (-0.5817066004316463,0.7267074909165057) -- (-0.5857646473946045,0.7303315145147442) -- (-0.5895325039899207,0.7340837695571326) -- (-0.5930040162759456,0.7379590438798698) -- (-0.5961730303110291,0.7419521253191543) -- (-0.5990333921535216,0.7460578017111854) -- (-0.6015789478617733,0.750270860892162) -- (-0.6038035434941346,0.7545860906982826) -- (-0.6057010251089555,0.7589982789657463) -- (-0.6072652387645866,0.7635022135307521) -- (-0.6084900305193777,0.7680926822294987) -- (-0.6093692464316793,0.772764472898185) -- (-0.6098967325598418,0.77751237337301) -- (-0.6100663349622153,0.7823311714901725) -- (-0.6098718996971499,0.7872156550858715) -- (-0.6093072728229961,0.7921606119963056) -- (-0.6083663003981039,0.7971608300576738) -- (-0.6070428284808238,0.8022110971061751) -- (-0.605330703129506,0.8073062009780083) -- (-0.6032237704025005,0.8124409295093725) -- (-0.6007158763581579,0.8176100705364662) -- (-0.5978008670548283,0.8228084118954885) -- (-0.5944725885508619,0.8280307414226382) -- (-0.5907248869046089,0.8332718469541143) -- (-0.5865516081744199,0.8385265163261155) -- (-0.5819465984186447,0.8437895373748407) -- (-0.5769037036956338,0.8490556979364892) -- (-0.5714167700637374,0.8543197858472594) -- (-0.5654796435813058,0.8595765889433502) -- (-0.5590861703066891,0.8648208950609608) -- (-0.5522301962982376,0.8700474920362898) -- (-0.5449055676143018,0.8752511677055363) -- (-0.5371061303132316,0.8804267099048989) -- (-0.5288257304533774,0.8855689064705767) -- (-0.5200582140930895,0.8906725452387685) -- (-0.5107974272907181,0.8957324140456734) -- (-0.5010372161046134,0.9007433007274899) -- (-0.49077142659312567,0.9056999931204172) -- (-0.47999390481460524,0.9105972790606539) -- (-0.46869849682740233,0.915429946384399) -- (-0.4568790486898672,0.9201927829278516) -- (-0.44452940646035,0.9248805765272103) -- (-0.4316434161972011,0.929488115018674) -- (-0.4182149239587706,0.9340101862384417) -- (-0.404237775803409,0.9384415780227126) -- (-0.38970581778946634,0.9427770782076849) -- (-0.374612895975293,0.9470114746295577) -- (-0.35895285641923913,0.9511395551245301) -- (-0.34271954517965497,0.955156107528801) -- (-0.32590680831489083,0.9590559196785691) -- (-0.30850849188329693,0.9628337794100331) -- (-0.29051844194322357,0.9664844745593923) -- (-0.27193050455302104,0.9700027929628456) -- (-0.2527385257710394,0.9733835224565914) -- (-0.23293635165562904,0.9766214508768291) -- (-0.21251782826514024,0.9797113660597571) -- (-0.19147680165792316,0.9826480558415747) -- (-0.16980711789232808,0.9854263080584804) -- (-0.1475026230267053,0.9880409105466736) -- (-0.12455716311940501,0.9904866511423526) -- (-0.10096458422877749,0.9927583176817165) -- (-0.07671873241317297,0.9948506980009645) -- (-0.05181345373094172,0.996758579936295) -- (-0.02624259424043398,0.9984767513239075) -- (0.,1.);
		\draw[line width=\lw] (2.,1.) -- (2.0084426374361324,0.9837748979219143) -- (2.016433504794426,0.9676909269705167) -- (2.023980192911968,0.9517508105337089) -- (2.0310902926258425,0.9359572719993925) -- (2.037771394773135,0.9203130347554688) -- (2.044031090190931,0.9048208221898394) -- (2.049876969716316,0.889483357690406) -- (2.0553166241863754,0.87430336464507) -- (2.0603576444381937,0.8592835664417329) -- (2.065007621308857,0.8444266864682963) -- (2.06927414563545,0.8297354481126618) -- (2.0731648082550587,0.8152125747627305) -- (2.0766872000047685,0.8008607898064045) -- (2.079848911721664,0.7866828166315852) -- (2.0826575342428315,0.7726813786261738) -- (2.0851206584053545,0.7588591991780722) -- (2.087245875046321,0.7452190016751817) -- (2.0890407750028137,0.7317635095054039) -- (2.09051294911192,0.7184954460566406) -- (2.091669988210724,0.7054175347167929) -- (2.0925194831363116,0.6925324988737624) -- (2.0930690247257675,0.6798430619154509) -- (2.093326203816178,0.6673519472297598) -- (2.0932986112446272,0.6550618782045904) -- (2.092993837848202,0.6429755782278447) -- (2.092419474463986,0.6310957706874238) -- (2.091583111929066,0.6194251789712296) -- (2.090492341080526,0.6079665264671632) -- (2.089154752755452,0.5967225365631266) -- (2.08757793779093,0.5856959326470209) -- (2.085769487024044,0.5748894381067478) -- (2.083736991291881,0.5643057763302088) -- (2.0814880414315247,0.5539476707053057) -- (2.079030228280061,0.5438178446199398) -- (2.0763711426745752,0.5339190214620126) -- (2.073518375452153,0.5242539246194257) -- (2.070479517449879,0.5148252774800806) -- (2.06726215950484,0.5056358034318789) -- (2.0638738924541196,0.496688225862722) -- (2.0603223071348036,0.4879852681605116) -- (2.056614994383978,0.479529653713149) -- (2.0527595450387275,0.47132410590853585) -- (2.0487635499361376,0.4633713481345737) -- (2.044634599913294,0.45567410377916423) -- (2.0403802858072813,0.4482350962302088) -- (2.0360081984551854,0.441057048875609) -- (2.0315259286940908,0.43414268510326615) -- (2.026941067361084,0.42749472830108204) -- (2.0222612052932503,0.42111590185695813) -- (2.0174939333276742,0.4150089291587959) -- (2.0126468423014408,0.40917653359449696) -- (2.0077275230516367,0.40362143855196286) -- (2.0027435664153463,0.398346367419095) -- (1.9977025632296548,0.393354043583795) -- (1.992612104331648,0.38864719043396434) -- (1.9874797805584112,0.38422853135750457) -- (1.9823131827470295,0.3801007897423173) -- (1.9771199017345884,0.376266688976304) -- (1.9719075283581735,0.3727289524473662) -- (1.9666836534548697,0.36949030354340534) -- (1.9614558678617624,0.36655346565232316) -- (1.9562317624159367,0.363921162162021) -- (1.9510189279544783,0.36159611646040046) -- (1.9458249553144724,0.3595810519353631) -- (1.9406574353330046,0.3578786919748104) -- (1.9355239588471598,0.35649175996664395) -- (1.9304321166940235,0.35542297929876526) -- (1.9253894997106813,0.35467507335907583) -- (1.9204036987342177,0.3542507655354772) -- (1.9154823046017193,0.35415277921587096) -- (1.9106329081502702,0.3543838377881585) -- (1.9058631002169566,0.35494666464024155) -- (1.9011804716388636,0.3558439831600214) -- (1.896592613253076,0.3570785167353998) -- (1.8921071158966796,0.35865298875427826) -- (1.8877315704067597,0.3605701226045582) -- (1.8834735676204017,0.3628326416741411) -- (1.8793406983746905,0.36544326935092875) -- (1.8753405535067122,0.3684047290228224) -- (1.8714807238535516,0.3717197440777238) -- (1.8677688002522939,0.3753910379035343) -- (1.8642123735400247,0.3794213338881556) -- (1.8608190345538294,0.38381335541948913) -- (1.8575963741307933,0.38856982588543637) -- (1.8545519831080015,0.3936934686738989) -- (1.8516934523225395,0.39918700717277833) -- (1.8490283726114924,0.4050531647699761) -- (1.8465643348119458,0.41129466485339383) -- (1.8443089297609854,0.417914230810933) -- (1.8422697482956956,0.4249145860304951) -- (1.840454381253162,0.43229845389998167) -- (1.8388704194704704,0.44006855780729426) -- (1.837525453784706,0.44822762114033443) -- (1.8364270750329537,0.4567783672870037) -- (1.8355828740522995,0.46572351963520375) -- (1.835000441679828,0.47506580157283573) -- (1.834687368752625,0.4848079364878014) -- (1.8346512461077757,0.49495264776800235) -- (1.8348996645823654,0.50550265880134) -- (1.8354402150134794,0.5164606929757161) -- (1.8362804882382031,0.5278294736790318) -- (1.837428075093622,0.539611724299189) -- (1.8388905664168214,0.5518101682240888) -- (1.8406755530448862,0.5644275288416333) -- (1.842790625814902,0.5774665295397237) -- (1.8452433755639543,0.5909298937062614) -- (1.8480413931291282,0.6048203447291481) -- (1.851192269347509,0.6191406059962855) -- (1.8547035950561819,0.6338934008955748) -- (1.8585829610922324,0.6490814528149178) -- (1.8628379582927466,0.6647074851422159) -- (1.8674761774948085,0.6807742212653706) -- (1.8725052095355044,0.6972843845722835) -- (1.877932645251919,0.714240698450856) -- (1.883766075481138,0.73164588628899) -- (1.8900130910602466,0.7495026714745865) -- (1.8966812828263304,0.7678137773955476) -- (1.9037782416164741,0.7865819274397743) -- (1.9113115582677636,0.8058098449951685) -- (1.9192888236172838,0.8255002534496315) -- (1.9277176285021205,0.845655876191065) -- (1.9366055637593587,0.8662794366073705) -- (1.945960220226084,0.8873736580864492) -- (1.9557891887393815,0.9089412640162033) -- (1.966100060136336,0.9309849777845338) -- (1.9769004252540343,0.9535075227793425) -- (1.98819787492956,0.9765116223885306) -- (2.,1.);
		\draw[line width=\lw] (3.,2.7320508075688776) -- (3.0182726692966844,2.7312497950244667) -- (3.0361972304134723,2.7301281036793226) -- (3.0537751202458012,2.728693669085148) -- (3.071007775689104,2.7269544267936454) -- (3.087896633638819,2.7249183123565177) -- (3.104443130990381,2.7225932613254678) -- (3.1206487046392244,2.7199872092521975) -- (3.1365147914807863,2.7171080916884103) -- (3.1520428284105013,2.7139638441858085) -- (3.1672342523238055,2.7105624022960937) -- (3.182090500116134,2.7069117015709714) -- (3.196613008682922,2.7030196775621413) -- (3.2108032149196073,2.698894265821307) -- (3.224662555721623,2.694543401900172) -- (3.238192467984405,2.6899750213504383) -- (3.2513943886033907,2.6851970597238077) -- (3.2642697544740136,2.6802174525719846) -- (3.2768200024917102,2.6750441354466705) -- (3.2890465695519167,2.6696850438995683) -- (3.300950892550067,2.664148113482381) -- (3.312534408381598,2.65844127974681) -- (3.3237985539419457,2.652572478244559) -- (3.3347447661265437,2.6465496445273318) -- (3.345374481830829,2.6403807141468283) -- (3.3556891379502374,2.634073622654753) -- (3.365690171380204,2.6276363056028083) -- (3.3753790190161648,2.6210766985426965) -- (3.384757117753554,2.6144027370261202) -- (3.393825904487809,2.6076223566047827) -- (3.4025868161143644,2.600743492830386) -- (3.4110412895286557,2.5937740812546335) -- (3.41919076162612,2.5867220574292267) -- (3.4270366693021894,2.5795953569058687) -- (3.434580449452303,2.5724019152362634) -- (3.441823538971895,2.5651496679721117) -- (3.448767374756401,2.557846550665117) -- (3.455413393701256,2.5505004988669815) -- (3.4617630327018967,2.5431194481294086) -- (3.467817728653758,2.535711334004101) -- (3.4735789184522767,2.5282840920427607) -- (3.4790480389928864,2.52084565779709) -- (3.4842265271710238,2.5134039668187924) -- (3.489115819882124,2.5059669546595704) -- (3.4937173540216233,2.4985425568711266) -- (3.4980325664849565,2.4911387090051633) -- (3.50206289416756,2.4837633466133835) -- (3.5058097739648684,2.47642440524749) -- (3.5092746427723176,2.4691298204591847) -- (3.512458937485344,2.4618875278001715) -- (3.5153640949993825,2.4547054628221523) -- (3.5179915522098684,2.447591561076829) -- (3.5203427460122376,2.4405537581159056) -- (3.5224191133019254,2.4335999894910842) -- (3.5242220909743693,2.426738190754067) -- (3.525753115925001,2.4199762974565577) -- (3.5270136250492596,2.413322245150258) -- (3.5280050552425797,2.4067839693868702) -- (3.528728843400396,2.4003694057180986) -- (3.5291864264181445,2.394086489695644) -- (3.5293792411912612,2.387943156871211) -- (3.529308724615182,2.3819473427965008) -- (3.528976313585341,2.3761069830232158) -- (3.528383444997175,2.37043001310306) -- (3.527531555746119,2.364924368587735) -- (3.5264220827276094,2.3595979850289437) -- (3.525056462837081,2.3544587979783884) -- (3.5234361329699695,2.3495147429877727) -- (3.5215625300217104,2.3447737556087986) -- (3.51943709088774,2.340243771393169) -- (3.517061252463493,2.3359327258925866) -- (3.5144364516444053,2.331848554658753) -- (3.5115641253259122,2.3279991932433726) -- (3.5084457104034508,2.3243925771981466) -- (3.505082643772454,2.3210366420747786) -- (3.501476362328359,2.3179393234249703) -- (3.4976283029666018,2.3151085568004253) -- (3.493539902582617,2.3125522777528458) -- (3.489212598071841,2.310278421833935) -- (3.4846478263297085,2.308294924595394) -- (3.4798470242516557,2.306609721588927) -- (3.474811628733118,2.3052307483662364) -- (3.469543076669531,2.3041659404790247) -- (3.46404280495633,2.303423233478994) -- (3.458312250488952,2.303010562917848) -- (3.4523528501628307,2.3029358643472877) -- (3.446166040873402,2.3032070733190175) -- (3.439753259516102,2.303832125384739) -- (3.4331159429863662,2.3048189560961556) -- (3.426255528179631,2.3061755010049696) -- (3.41917345199133,2.3079096956628833) -- (3.4118711513169,2.3100294756215995) -- (3.404350063051777,2.3125427764328212) -- (3.3966116240913955,2.315457533648251) -- (3.3886572713311924,2.318781682819591) -- (3.3804884416666017,2.3225231594985445) -- (3.3721065719930596,2.326689899236814) -- (3.3635130992060027,2.3312898375861018) -- (3.3547094602008656,2.3363309100981113) -- (3.3456970918730837,2.3418210523245446) -- (3.3364774311180936,2.3477681998171036) -- (3.327051914831329,2.3541802881274925) -- (3.317421979908228,2.3610652528074123) -- (3.307589063244224,2.3684310294085678) -- (3.2975546017347535,2.376285553482659) -- (3.287320032275252,2.384636760581391) -- (3.276886791761156,2.3934925862564644) -- (3.266256317087899,2.4028609660595843) -- (3.255430045150918,2.412749835542451) -- (3.2444094128456484,2.4231671302567674) -- (3.233195857067526,2.4341207857542377) -- (3.2217908147119863,2.445618737586563) -- (3.2101957226744644,2.457668921305447) -- (3.1984120178503956,2.4702792724625917) -- (3.186441137135216,2.483457726609701) -- (3.174284517424362,2.4972122192984743) -- (3.161943595613268,2.5115506860806187) -- (3.14941980859737,2.526481062507834) -- (3.136714593272103,2.542011284131823) -- (3.123829386532903,2.5581492865042894) -- (3.110765625275207,2.5749030051769344) -- (3.097524746394449,2.592280375701462) -- (3.084108186786064,2.6102893336295745) -- (3.070517383345489,2.628937814512974) -- (3.056753772968159,2.6482337539033645) -- (3.042818792549509,2.6681850873524473) -- (3.0287138789849757,2.688799750411925) -- (3.0144404691699944,2.710085678633501) -- (3.,2.7320508075688776);
		\draw[line width=\lw] (2.,4.464101615137755) -- (2.0098300318605524,4.479525704671429) -- (2.0197637256190464,4.494487984277684) -- (2.029794927333833,4.508993666120317) -- (2.039917483063262,4.523047962363131) -- (2.050125238865684,4.536656085169927) -- (2.0604120407994495,4.5498232467045066) -- (2.070771734922909,4.562554659130669) -- (2.0811981672944113,4.574855534612218) -- (2.091685183972308,4.586731085312953) -- (2.1022266310149487,4.5981865233966746) -- (2.112816354480684,4.609227061027187) -- (2.123448200427864,4.619857910368288) -- (2.134116014914839,4.6300842835837805) -- (2.144813643999959,4.639911392837465) -- (2.155534933741575,4.6493444502931425) -- (2.166273730198036,4.658388668114614) -- (2.177023879427693,4.66704925846568) -- (2.187779227488897,4.675331433510145) -- (2.1985336204399966,4.683240405411806) -- (2.2092809043393435,4.690781386334466) -- (2.2200149252452874,4.697959588441926) -- (2.230729529216178,4.704780223897986) -- (2.241418562310366,4.711248504866449) -- (2.2520758705862023,4.717369643511116) -- (2.2626953001020365,4.723148851995786) -- (2.2732706969162186,4.728591342484262) -- (2.2837959070870997,4.733702327140345) -- (2.294264776673029,4.738487018127835) -- (2.3046711517323573,4.742950627610534) -- (2.315008878323435,4.747098367752243) -- (2.325271802504612,4.750935450716764) -- (2.3354537703342384,4.754467088667896) -- (2.3455486278706656,4.757698493769441) -- (2.3555502211722423,4.760634878185202) -- (2.36545239629732,4.763281454078976) -- (2.3752489993042483,4.765643433614569) -- (2.3849338762513774,4.767726028955779) -- (2.3945008731970576,4.769534452266408) -- (2.4039438361996397,4.771073915710256) -- (2.413256611317473,4.772349631451126) -- (2.4224330446089084,4.773366811652819) -- (2.4314669821322967,4.774130668479135) -- (2.4403522699459868,4.7746464140938745) -- (2.44908275410833,4.77491926066084) -- (2.4576522806776757,4.774954420343832) -- (2.4660546957123746,4.774757105306652) -- (2.474283845270778,4.774332527713101) -- (2.482333575411234,4.773685899726981) -- (2.4901977321920943,4.772822433512092) -- (2.4978701616717087,4.771747341232234) -- (2.505344709908428,4.77046583505121) -- (2.512615222960602,4.76898312713282) -- (2.5196755468865804,4.767304429640867) -- (2.5265195277447146,4.76543495473915) -- (2.5331410115933544,4.763379914591471) -- (2.5395338444908493,4.761144521361631) -- (2.54569187249555,4.75873398721343) -- (2.5516089416658074,4.7561535243106725) -- (2.5572788980599723,4.7534083448171565) -- (2.5626955877363926,4.750503660896683) -- (2.5678528567534205,4.747444684713056) -- (2.572744551169405,4.744236628430071) -- (2.577364517042697,4.7408847042115365) -- (2.5817066004316476,4.737394124221249) -- (2.5857646473946057,4.733770100623012) -- (2.589532503989922,4.7300178455806225) -- (2.5930040162759465,4.726142571257885) -- (2.59617303031103,4.722149489818601) -- (2.5990333921535225,4.718043813426569) -- (2.6015789478617743,4.713830754245593) -- (2.603803543494135,4.709515524439473) -- (2.6057010251089565,4.705103336172009) -- (2.6072652387645876,4.700599401607003) -- (2.6084900305193788,4.696008932908256) -- (2.6093692464316804,4.69133714223957) -- (2.609896732559843,4.686589241764745) -- (2.610066334962216,4.681770443647583) -- (2.6098718996971506,4.676885960051884) -- (2.6093072728229973,4.67194100314145) -- (2.608366300398105,4.666940785080081) -- (2.6070428284808247,4.66189051803158) -- (2.605330703129507,4.656795414159747) -- (2.6032237704025016,4.651660685628382) -- (2.600715876358159,4.6464915446012895) -- (2.597800867054829,4.641293203242266) -- (2.594472588550863,4.636070873715117) -- (2.59072488690461,4.6308297681836414) -- (2.5865516081744206,4.625575098811639) -- (2.5819465984186456,4.6203120777629145) -- (2.576903703695635,4.615045917201266) -- (2.5714167700637383,4.609781829290496) -- (2.565479643581307,4.604525026194404) -- (2.5590861703066903,4.599280720076795) -- (2.5522301962982388,4.594054123101466) -- (2.5449055676143026,4.588850447432218) -- (2.5371061303132323,4.583674905232857) -- (2.528825730453378,4.578532708667177) -- (2.52005821409309,4.573429069898986) -- (2.510797427290719,4.568369201092082) -- (2.501037216104614,4.563358314410266) -- (2.4907714265931267,4.558401622017338) -- (2.4799939048146062,4.553504336077101) -- (2.468698496827403,4.5486716687533555) -- (2.456879048689868,4.543908832209904) -- (2.444529406460351,4.539221038610545) -- (2.431643416197202,4.534613500119081) -- (2.418214923958771,4.530091428899313) -- (2.4042377758034097,4.525660037115043) -- (2.3897058177894666,4.52132453693007) -- (2.374612895975294,4.517090140508198) -- (2.35895285641924,4.5129620600132245) -- (2.3427195451796554,4.508945507608954) -- (2.3259068083148913,4.505045695459187) -- (2.3085084918832974,4.501267835727722) -- (2.290518441943224,4.497617140578363) -- (2.2719305045530214,4.49409882217491) -- (2.2527385257710395,4.490718092681163) -- (2.2329363516556295,4.487480164260926) -- (2.2125178282651405,4.484390249077998) -- (2.1914768016579234,4.481453559296181) -- (2.1698071178923284,4.478675307079275) -- (2.1475026230267056,4.476060704591082) -- (2.124557163119405,4.473614963995403) -- (2.1009645842287776,4.4713432974560385) -- (2.0767187324131733,4.469250917136792) -- (2.051813453730942,4.467343035201461) -- (2.026242594240434,4.465624863813847) -- (2.,4.464101615137755);
		\begin{scriptsize}
		\draw [fill=black] (0.,1.) circle (2.5pt);
		\draw [fill=black] (2.,1.) circle (2.5pt);
		\draw [fill=black] (3.,2.7320508075688776) circle (2.5pt);
		\draw [fill=black] (2.,4.464101615137755) circle (2.5pt);
		\draw [fill=black] (0.,4.464101615137755) circle (2.5pt);
		\draw [fill=black] (-1.,2.732050807568879) circle (2.5pt);
		\draw [fill=black] (1.6240304368012277,-5.3014430888370185) circle (2.5pt);
		\draw[color=black] (-1.6046226324263282,5.214100588022071) node {$a = 0.6$};
		\end{scriptsize}
		\end{tikzpicture}
		\caption{Possible changes in the class when adding an angle of class $1$, $3$, $5$ or $7$.}
		\label{fig:types}
	\end{figure}
		
	Now we are ready to finish the proof. Let $T$ be the trail of length $n+1$ in $G$ whose $\ell$-th vertex is the class of $\sum\limits_{i=1}^{\ell-1} \phi_i $. We know that  $\sum\limits_{i=1}^n \phi_i  $ is a integer multiple of $2\pi$, so  $\cos(\sum\limits_{i=1}^n \phi_i)=\frac{2}{2\cdot 1}$.  Hence the trail should start and end at 2. By Lemma \ref{lemma:classadition} when the class of $\phi_i$ is 1 or 5, we follow one of the solid edges, if the class is 3 or 7, we follow a dashed edge. 
	
	Finally, we show that we followed a dashed edge an odd number of times. Considering the equation $(r_1r_2)(r_2r_3)\cdots(r_{n-1}r_{n})(r_{n}r_1)=(\prod r_i)^2\equiv 1 \mod 4$ we have $r_ir_{i+1}\equiv 3 \mod 4$ for an even number of $i$-s. Since $n$ is odd, this implies that we have an odd number of $i$-s when $r_ir_{i+1}\equiv 1 \mod 4$. Hence the trail contains an odd number of dashed edges. Since the dashed edges form a bipartite graph and the solid edges connect vertices inside the two parts the trail cannot end where it started, a contradiction. This shows that a counterexample to Theorem \ref{thm:mainwheel} cannot exists.

\end{proof}

\section{Final remarks}

We note that some parts of the proof can be replaced by other arguments. For example Lemma \ref{lemma:classadition} also follows from the analysis of Cayley-Menger determinants.  

The main goal of understanding odd-distance graphs is to determine the chromatic number of $\godd$. Odd wheels are the simplest graphs that are not 3-colorable, yet they are not odd-distance graphs. Our proof heavily relies on the fact that a wheel graph contains many triangles. An other nice question of Rosenfeld and Nam L\^{e} Tien \cite{MR3159068} is the following. Are there triangle-free graphs that are not odd-distance graphs?

Piepemeyer's construction which shows that $K_{n,n,n}$ is an odd-distance graph comes from an integral point set. Naturally, one might be tempted to look for odd-distance graphs with high chromatic number in a similar way. Take an integral point set and then consider the odd-distance graph given by the edges of odd length. We note that this method cannot lead to success since the chromatic number of these graphs is at most 3. We leave the proof of this statement to the interested readers.  
 
    We can also consider the natural analog of Harborth's conjecture. Which planar graphs have a planar drawing where the length of the edges are odd integers?
    
    Take for example a maximal planar graph, in other words a triangulation. If it contains an odd wheel, it is not an odd-distance graph. On the other hand if it does not contain an odd wheel, it is 3-colorable. Hence it is an odd distance graph, but this does not imply that we can find an plane drawing without crossings.  Is it true that all 3-colorable planar graphs have an embedding without crossings where the length of the edges are odd integers?

\section{Acknowledgement}	
We  would like to thank	SciExperts for providing free access to the software Wolfram Mathematica, and therefore to the database of Ed Pegg Jr. \cite{eddpeg} on embeddings of wheels. We also thank D\"om\"ot\"or P\'alv\"olgyi and our anonymous reviewers for valuable suggestions and encouragement.

 \bibliographystyle{splncs04}
 \bibliography{wheels}
\end{document}